\documentclass[11pt]{amsart}

\usepackage{amssymb,amsmath,enumerate,epsfig}
\usepackage{amsfonts,color}
\usepackage{a4,latexsym}
\usepackage{parskip}

\usepackage{hyperref}
\hypersetup{pdftitle=Barth5}

\newcommand{\Sym}{\mathfrak{S}}
\newcommand{\Csym}{\mathfrak{C}}
\newcommand{\Slambda}{S}
\newcommand{\C} {\mathbb{C}}
\newcommand{\Q} {\mathbb{Q}}
\newcommand{\N}  {\mathbb{N}}

\newcommand{\F}{\mathbb{F}}
\newcommand{\Z}{\mathbb{Z}}
\newcommand{\p}{\mathfrak{p}}

\newcommand{\PP}{\mathbb{P}}
\newcommand{\NS}{\mathop{\rm NS}}
\newcommand{\Num}{\mathop{\rm Num}}

\newcommand{\Aut}{\mathop{\rm Aut}}

\newtheorem{Theorem}{Theorem}[section]
\newtheorem{Proposition}[Theorem]{Proposition}
\newtheorem{Lemma}[Theorem]{Lemma}

\newtheorem{Corollary}[Theorem]{Corollary}

\theoremstyle{remark}
\newtheorem{Remark}[Theorem]{Remark}

\theoremstyle{definition}

\begin{document}

\title[The Barth quintic surface has Picard number 41]{The Barth quintic surface has Picard number 41}

\dedicatory{Dedicated to Wolf Barth}

\author{Slawomir Rams}
\address{Institute of Mathematics , Jagiellonian University, 
ul. {\L}ojasiewicza 6,  30-348 Krak\'ow, Poland} 
\email{slawomir.rams@uj.edu.pl}

\author{Matthias Sch\"utt}
\address{Institut f\"ur Algebraische Geometrie, Leibniz Universit\"at
  Hannover, Welfengarten 1, 30167 Hannover, Germany}
\email{schuett@math.uni-hannover.de}

\subjclass[2010]{14J29; 11G25, 14C22, 14G15, 14J27, 14J28}
\keywords{Surface of general type, lines, Picard number, K3 surface, Godeaux surface, Tate conjecture}
\thanks{Partial funding   by  MNiSW grant no. N N201 388834
and by ERC StG~279723 (SURFARI)
 is gratefully acknowledged}
%
%

\date{July 1, 2012}

 \begin{abstract}
This paper investigates a specific smooth quintic surface suggested by Barth for it contains the current record of 75 lines over the complex numbers. Our main incentive is to prove that the complex quintic has Picard number 41, and to compute the N\'eron-Severi group up to a $2$-power index. We also compute Picard numbers for reductions to positive characteristic and verify the Tate conjecture.
 \end{abstract}
 \maketitle


%

\section{Introduction}

Quintic surfaces in $\PP_3$ have been studied extensively by Barth and others,
for instance with a view towards configurations of singularities or lines contained in them.
This paper investigates a specific smooth quintic surface suggested by Barth 
for it contains the current record of 75 lines over $\C$ (see also \cite{xie2010}).
In what follows the surface will be denoted by $S_a$.
Our main incentive is to prove that over $\C$ the quintic $S_a$ has Picard number 41
(Theorem \ref{thm}).
To the best of our knowledge this is the record Picard number for smooth quintics.
In fact the surfaces with Picard number 43 or 45 exhibited in \cite{S-quintic}
involved several rational double point singularities.
The previous record of 37 was attained by the Fermat quintic surface
which also contains 75 lines (Remark \ref{rem:5}).

This note is organised as follows.
Section \ref{s:surf}  reviews the surface $S_a$ inside a pencil of quintics
with an action of the symmetric group $\Sym_5$.
Sections \ref{s:Godeaux} and \ref{s:K3} derive lower and upper bounds for the Picard number 
by exhibiting certain quotient surfaces (Godeaux and K3).
As a by-product we prove the Tate conjecture for any irreducible member of the pencil of quintics 
(cf.~Proposition \ref{prop:tate}).
Throughout we keep the exposition as characteristic free as possible.
This enables us to apply the supersingular reduction technique from \cite{SSvL}
to calculate the N\'eron-Severi group of $S_a$ up to index $2^i$ for some $i\leq 4$
in Section \ref{s:NS}.
We also work out an explicit non-classical Godeaux surface (Proposition \ref{prop:5})
compared to Miranda's implicit results in \cite{Miranda}.

\section{A pencil of $\Sym_5$-invariant quintics in $\PP_3$} \label{sect-pencil}
\label{s:surf}

In this note we consider certain surfaces that belong to the pencil of quintics
\begin{equation} \label{eq-pencil}
S_{\lambda}  \quad :  \quad  \left\{s_1 = \frac 56 \lambda s_2 \cdot s_3  + s_5   = 0\right\}\subset\PP_4  ,  \quad \lambda \in K \, ,  
\end{equation}
where $s_k$ stands for the symmetric polynomial
$$
s_k := x_0^k + x_1^k + x_2^k + x_3^k + x_4^k \;\;\;\;\; (k\in\N)
$$ 
and $K$ denotes an algebraically closed field of any characteristic. 
Mostly we will be concerned with the case $K=\C$,
but our methods to investigate these surfaces will use reduction modulo different primes,
and in fact we will also derive results exclusive to positive characteristic.
The factor of $5/6$ in front of $\lambda$ might seem unnatural at first,
but in fact it allows us to derive proper pencils in characteristics $2,3,5$
by substituting $s_1$ in the quintic polynomial and eliminating common factors.
It should be understood that we always work with such a proper model of the pencil in the sequel.

The above pencil (albeit without the extra factor) was  studied by Barth  in order to find a quintic with 15 three-divisible cusps (\cite{barth99}) and 
smooth quintics with 
many lines (\cite{barth2000}).
For the convenience of the reader we list below the facts from \cite{barth99}, \cite{barth2000} 
that we will use in the sequel.
All of them can be verified by straightforward computation (possibly with help of a computer program);
related properties also appear in \cite{xie2010} (see Remark \ref{rem-sing}).

Observe that if we denote by $B_{10}$ (resp.~$B_{15}$) the curve in $\PP_4$ given by $s_1 = s_5 = s_2 = 0$ (resp.~$s_1 = s_5 = s_3 = 0$), then  
the base locus of the pencil in question is the curve $B_{10} \cup B_{15}$. One can check 
by direct computation that 
the curve $B_{15}$ consists of the $15$ lines
\begin{equation}   \label{eq-B15line}
x_{i_1} = x_{i_2} + x_{i_3} = x_{i_4} + x_{i_5} = 0,  
\end{equation} 
where $i_1, i_2, i_3, i_4, i_5$ are pairwise different, i.e.~$\{i_1, i_2, i_3, i_4, i_5\}=\{0,1,2,3,4\}$. Similarly, the curve $B_{10}$ is the union of the 
five  conics $C_{i_1}$ (smooth outside characteristic $2$)
\begin{equation}   \label{eq-conic} 
x_{i_1} = x_{i_2}^2 + x_{i_3}^2 + x_{i_4}^2 +  x_{i_2}  x_{i_3} + x_{i_2} x_{i_4} + x_{i_3} x_{i_4}  = s_1 = 0.
\end{equation}
Therefore, the plane $x_{i_1} = s_1 = 0$ meets the base locus along the three lines \eqref{eq-B15line}
and the  conic \eqref{eq-conic}. In particular, the four curves 
are the only components  of intersection of the plane   $x_{i_1} = s_1 = 0$ with an irreducible quintic  
$S$ that belongs to the pencil.    
The general member  of the pencil $\{ S_\lambda \}$ is smooth.
Outside characteristics $3,13,17$
this can be checked with the Jacobi criterion aplied to the special member at $\lambda=0$.
Below we work out all singular values of $\lambda$ following \cite{barth2000},  \cite{xie2010}
independent of the characteristic.

\begin{Lemma}
\label{lem:smooth}
$S_\lambda$ is non-smooth exactly for $\lambda \in\left\{  -1, -\frac{3}{2}, -\frac{51}{50}, -\frac{13}{25}, -\frac{1}{2}, \infty\right\}$.
\end{Lemma}

\begin{proof}
Over $\C$ the lemma has been proved in \cite[\S3]{xie2010}.
Using this it  is easy to check that $S_\lambda$ is singular at the given values for $\lambda$
independent of the characteristic.
Thus it remains to prove that for neither $p$ there are other singular parameters $\lambda$
modulo $p$.
For this purpose, we follow the argument in \cite{xie2010} fairly closely and switch to the representation of our pencil
as hypersurface $\{s_1=0\}$ in the Dwork pencil:
\[
F_t\; = \; s_5-5tx_0\cdots x_4.
\]
Intersected with $\{s_1=0\}$, this reduces to $\frac 56 s_2\cdot s_3=0$ at $t=1$,
so the parameters $\lambda, t$ are related by
\[
\lambda = \frac t{1-t},\;\;\; t = \frac\lambda{\lambda+1}.
\]
The corresponding singular values for $t$  are $\left\{-1,1,3,51,-\frac{13}{12},\infty\right\}$.
By the Jacobi criterion $S_\lambda$ is singular if  all partial derivatives of $F_t$ are equal.
Independent of the characteristic, it suffices to consider the partials of $s_1$ and
\[
x_i^4 - t \prod_{j\neq i} x_j = \partial_{x_i}\left(\frac{s_5-s_1^5-5tx_0\cdots x_4}5\right)\mid_{s_1=0} \;\;\; (i=0,\hdots,4).
\]
In  case of vanishing partial derivatives, one easily rules out $x_i=0$ for all $i=0,\hdots,4$.
In the affine chart $x_0=1$ the partials then return 
\[
1 = t x_1\cdots x_4,\;\;\; x_i^4 = t \prod_{j\neq i} x_j \;\;\; (i=1,\hdots,4)
\]
which readily implies $x_i^5=1$ for each index $i$. 
Over $\C$, the hypersurface condition $s_1=0$ then requires
 that the $x_i$ run through the full set of fifth roots of unity,
so that in particular $t=1$, one of the singular parameters.
In characteristic $p\equiv 2,3\mod 5$
the same argument goes through without modification.
For $p\equiv 4\mod 5$, one uses the fact that for a primitive fifth root of unity $\zeta\in\F_{p^2}$
we have $\zeta+\zeta^{-1}\in\F_p$.
Hence the condition $s_1=0$ breaks down into 4 cases
which can be shown to give no solution for any $p$ except for the above ones.
Finally for $p\equiv 1\mod 5$,
a case-by-case analysis using the norm of $\Q(\zeta)/\Q$ 
reveals three possibilities of additional singularities:
they occur at
$t=3, 9$ for $p=11$ and  at $t={10}$ for $p=41$.
By inspection, each value arises from some singular parameter over $\Q$ via reduction,
so there are no additional singular parameters modulo $p$.
This settles the case of all partial derivatives vanishing.

The case of equal non-zero partial derivatives
can be treated completely analogously to \cite{xie2010}.
For shortness we omit the details,
but it is easily checked that there are no further singular parameters.
%
%
\end{proof}



For the special quintic with 75 lines we introduce the following notation:   
\begin{eqnarray}
\label{eq:a}
a &:=&  -\frac{2}{b+2}, \quad \mbox{ where } \; b^4 - b^3 + 1 = 0.
\end{eqnarray}
Throughout the paper  $S_a$ stands for the surface given by \eqref{eq-pencil} with $\lambda = a$
(over $\C$ unless specified otherwise).
In the sequel we shall often write $\Slambda$ instead of $S_{\lambda}$ when
there is no ambiguity from the context.
%
%
By Lemma \ref{lem:smooth}, the surface $S_{a}$  is smooth  over $\C$
(for positive characteristic see Corollary \ref{cor:bad}).
One directly verifies that $S_a$ contains the line 
\begin{equation} \label{eq-extraline}
\mbox{span}(\{(1:-1:b:-b:0),((b-1):1:-(b-1):0:-1)\}),
\end{equation} 
with $b$ given by \eqref{eq:a}. 
In fact, the $\Sym_5$-action endows $S_{a}$  
with $60$ lines
obtained from \eqref{eq-extraline}
by virtue of the symmetries.
With the 75 complex lines at hand, 
we already have a good portion of divisors on $S_a$.
Our main result for this paper is:

\begin{Theorem}
\label{thm}
Over $\C$, the quintic $S_a$ has Picard number 41.
\end{Theorem}

\begin{Remark}
\label{rem:5}
To the best of our knowledge,
the Picard number 41 of $S_a$ gives a new record among smooth complex quintics.
In comparison, Picard numbers 43 and 45 have so far only been realised by desingularisations of quintics
with rational double point singularities in \cite{S-quintic}.
The previous record of 37 was attained by the Fermat quintic,
so Theorem \ref{thm} also gives an alternative way to see
that $S_a$ and the Fermat quintic surface cannot be isomorphic over $\C$.
In fact the surfaces differ also in another respect:
the Fermat quintic has $\NS$ generated by lines (even over $\Z$ by \cite{SSvL})
while any basis of $\NS(S_a)$ includes some other divisor class
that will be made visible in the proof of Lemma \ref{lem-lowerbound}.
\end{Remark}

The proof of Theorem \ref{thm} proceeds in two steps.
First we derive the lower bound $\rho(S_a)\geq 41$ 
by exhibiting a suitable quotient surface $Q$ of $S$ by a cyclic group of order 5
(a Godeaux surface studied in section \ref{s:Godeaux}).
Then we establish the upper bound $\rho(S_a)\leq 41$
through a quotient surface $X$ that is K3 in section \ref{s:K3}.
Here we use reduction modulo different primes 
and the Artin-Tate conjecture in a technique following van Luijk \cite{vL}
and Kloosterman \cite{RK}.

\begin{Remark} \label{rem-sing}
For $K=\C$ the pencil  $\{S_{\lambda}\}$ has also been studied in  \cite{xie2010}.
By \cite[Thm~1.2]{xie2010}
the pencil in question
 contains (up to Galois conjugation) 
exactly three smooth surfaces that carry a line other than those from $B_{15}$
 given by \eqref{eq-extraline}. 
Moreover, no quintic in the pencil \eqref{eq-pencil} contains more than $75$ lines
and the surface $S_{a}$ is  the unique (up to Galois conjugation) element  of the pencil which carries the maximal number of lines.
We will not use that result in the sequel, but it certainly motivates our interest in the quintic  $S_{a}$.
%
\end{Remark}

\section{Lower bound -- Godeaux quotient}
\label{s:Godeaux}



In this section we derive the lower bound $\rho(S_a)\geq 41$ (Lemma \ref{lem-lowerbound}).
At first   we exhibit a Godeaux surface $Q$
that arises from $S$ as a quotient by a cyclic group of order 5
acting without fixed points. 
Then a close examination of the 75 lines on $S_a$ and their images under the quotient map
implies the inequality in question.

Consider the automorphism $R : \PP_4 \rightarrow \PP_4$ defined as
$$
R(x_0 : x_1: x_2: x_3: x_4) := (x_4: x_0 : x_1: x_2: x_3).
$$
Outside characteristic $5$,
$R$ has five fixed points: 
$(1 : \varepsilon_5^k :  \varepsilon_5^{2k} :   \varepsilon_5^{3k} :   \varepsilon_5^{4k})$
where $k \in \{ 0, 1, 2, 3, 4\}$ and 
$\varepsilon_5 \neq 1$ is a root of unity of order five.  
Clearly each member of the pencil $\{S_\lambda\}$ is invariant under $R$,
so we can restrict $R$ to $S_\lambda$ and 
compute the fixed points. 
One directly sees that 
$s_l(1, \ldots, \varepsilon_5^{4k}) = 0$ for all $k\neq 0$ and $5\nmid l$, 
whereas $s_{5l}(1, \ldots, \varepsilon_5^{4k})=s_l(1,\hdots,1)=5$ 
for each $k$ and $l$.
In conclusion 
none of the fixed points of $R$ belong to $S_\lambda$
for any $\lambda \in K$. 
In characteristic $5$, there is only one fixed point $(1:1:1:1:1)$
which is easily verified to lie outside any quintic $S_\lambda$:
upon subsituting $s_1$ for $x_4$,
the relevant polynomial modulo $5$ is
\[
\frac{x_0^5+\hdots+x_3^5-(x_0+\hdots+x_3)^5}5
\]
which evaluates as $-4\cdot 51\not\equiv 0\mod 5$.

The automorphism $R$ generates a subgroup $\Csym_5 \subset \Sym_5\subset\Aut(S)$.
Assume that the quintic $\Slambda$ 
is smooth (or replace it by the minimal desingularisation if it is non-degenerate,
i.e.~it has only rational double points as singularities), 
then
\begin{equation}
\mbox{the quotient surface } Q
:= \Slambda/_{\Csym_5} \mbox{ is smooth.} 
\end{equation} 
We thus obtain a Godeaux surface.
If $\mbox{char}(K) \neq 5$, we can almost verbatim repeat the considerations of
\cite[Example~9.6.2]{badescu} to show that $Q$ 
is a minimal surface of general type with $\mbox{Pic}^\tau(Q)\cong \Z/5\Z$ and the following invariants:
\begin{equation} \label{eq-irregularity} 
h^1({\mathcal O}_Q) = h^2({\mathcal O}_Q) = 0,  \mbox{ and } {\mbox{K}^2_{Q}} = 1.
\end{equation}
In characteristic $5$, however, the invariants differ as $Q$ is a non-classical Godeaux surface
with $\mbox{Pic}^\tau(Q)\cong \mu_5$. Namely, because $\Csym_5\cong\Z/5\Z$, one finds as in \cite{Miranda}
\[
h^1({\mathcal O}_Q) = h^2({\mathcal O}_Q) =  {\mbox{K}^2_{Q}} = 1.
\]
Remember that $S_0$ has a smooth model in characteristic $5$.
As opposed to the implicit result of \cite{Miranda},
this yields an explicit non-classical Godeaux surface in characteristic $5$:

\begin{Proposition}
\label{prop:5}
In characteristic five, $Q_0$ is a non-classical Godeaux surface.
\end{Proposition}


We shall now turn to the Picard group of the special quintic $S_a$.
Our previous considerations put us in the position to derive a geometric lower bound for the Picard number.
We state the result here only over $\C$.
The argument goes through without essential modifications in positive characteristic as well,
but there we will derive better bounds in conjunction with the Tate conjecture (see Remark \ref{rem} and Corollary \ref{cor:tate}).

\begin{Lemma} \label{lem-lowerbound}
Over $\C$ the following inequality holds  
\begin{equation*} 
\rho(S_{a}) \geq 41.
\end{equation*}
\end{Lemma}

\begin{proof}
The surface $S_{a}$ 
carries the $15$ lines \eqref{eq-B15line}. Moreover, it contains the 
$60$ lines obtained by the action of symmetries
on the line \eqref{eq-extraline}. Let $M$  
be the Gram matrix of the 
$75$ lines in question. By direct computation we obtain
$$
\mbox{rank}(M) = 40 \;\;\; \text{ and hence } \;\;\; \rho(S_a)\geq 40.
$$
Observe that both $\mathcal{O}_{S_{a}}(3)$ and the divisors $3 C_i$, where $i= 0, \ldots, 4$,
lie in the span of the $75$ lines (see section~\ref{sect-pencil}),
so there are no other obvious independent curves on $S_a$.

In order to prove that actually $\rho(S_a)\geq 41$,
we consider the Godeaux surface $Q_a$ and the quotient map
\[
{\mathbb \pi} \, : \, S_{a} \rightarrow Q_{a}.
\]
By  \eqref{eq-irregularity}  and Noether's formula, we compute the 
topological Euler number (or Euler-Poincar\'e characteristic)  $\mbox{e}(Q_{a}) = 11$.  
Since we work over $\C$,  equality \eqref{eq-irregularity} 
implies in terms of the Hodge decomposition
\[
b_2(Q_a) = h^{1,1}(Q_a) = 9.
\]
Then Lefschetz' theorem on $(1,1)$-classes yields
 \begin{equation*} \label{eq-picardquotient}
 \rho(Q_{a}) =  b_2(Q_a) = 9 .
 \end{equation*}
 Pulling back divisors to $S_a$ via $\pi^*$, we find that
 \begin{eqnarray}\label{eq:inv}
 \mbox{rank}\left(\NS(S_a)^{R^*=1}\right) = 9.
 \end{eqnarray}
We shall now compare with the contribution from the lines which come in 15 $R$-orbits.
Let $L, L'$ be two of the $75$ lines on $S_{a}$. 
Then intersection numbers on $Q_a$ are given by
$$
{\mathbb \pi}(L).{\mathbb \pi}(L') = \frac{1}{5} \left(\sum_{i=0}^{4} R^i(L)\right).\left(\sum_{i=0}^{4} R^i(L')\right)
=
L.\left(\sum_{i=0}^{4} R^i(L')\right)
$$
Thus we can  compute the Gram matrix $N$
of the 15 divisors on $Q_a$ that are the 
images of the $75$ lines on $S_a$ under the quotient map $\pi$.  
A direct computation gives
\begin{equation} \label{eq-rankquotient}
\mbox{rank}(N) =8,\;\;\; \mbox{disc}(N) = -2.
\end{equation}
Comparing with \eqref{eq:inv}
we deduce that
there is an $R^*$-invariant divisor class in $D\in\NS(S_a)$ 
which is not contained in the $\Q$-span of the 75 lines.
This implies that $\rho(S_a) \geq 41$ as claimed.
\end{proof}

\begin{Remark}
\label{rem}
a)
In positive characteristics where $S_a$ is smooth,
exactly the same argument goes through 
after lifting $Q_a$  to $\C$ which implies the analogous (in)equalities
(or use reduction modulo $\p$).
Those characteristics where $S_a$ attains singularities require some extra care.\\
b)
K3 quotients and the Tate conjecture will allow us to derive better, and in fact precise estimates for the
Picard numbers  $\rho(S_a\otimes\bar\F_p)$ regardless of the (rational double point) singularities (Corollary \ref{cor:tate}, Remark \ref{rem:tate}).
\end{Remark}

\begin{Remark}
Alternatively, one could argue with the induced action of $R$ on the holomorphic 2-form over $\C$.
As we will infer in \eqref{eq:V^4}, this implies that the transcendental lattice $T(S)$ generally has 4-divisible rank.
Consequently $\rho(S)\equiv 1\mod 4$, so that for $S_a$ our lower bound $\rho\geq 40$ coming from the lines on $S_a$ automatically improves to the bound of Lemma \ref{lem-lowerbound}.
In our eyes, the given proof has two advantages:
relative independence of the characteristic (as sketched in Remark \ref{rem}.a))
and  a constructive nature which we will exploit in some detail in Section \ref{s:NS}
in order to compute $\NS(S_a)$ up to index $2^i, i\leq 4$.
\end{Remark}

\section{Upper bound -- quotient K3 surface}
\label{s:K3}

\subsection{}
\label{ss:quot}

In order to complete the proof of Theorem \ref{thm}
it remains to establish the upper bound $\rho(S_a)\leq 41$ over $\C$.
Here we shall crucially use the $\Sym_5$ action on the complex surface $S=S_\lambda$.
Consider the transcendental part $T(S)$ of $H^2(S)$
obtained as the orthogonal complement of $\NS(S)$ with respect to the intersection pairing
given by cup product.
This can be understood as a lattice, as a Hodge structure or as a Galois representation.
The Hodge structure is directly related to the regular 2-forms on $S$,
hence we study the action of $\Sym_5$ on $H^{2,0}(S)$.
Using the isomorphism (over any field $k$)
\[
H^{2,0}(S) \cong H^0(K_S) = H^0(\mathcal O_S(1)) \cong k[x_0,\hdots,x_3]^{(1)} \cong k[x_0,\hdots,x_4]^{(1)}/k s_1,
\]
one easily finds the irreducible representation of $\Sym_5$ on $H^{2,0}(S)$ given by
\begin{eqnarray}
\label{eq:sigma}
\Sym_5\ni \sigma: k[x_0,\hdots,x_4]^{(1)}/k s_1 & \to & k[x_0,\hdots,x_4]^{(1)}/k s_1\\
f \;\;\;\;\;\;\;\;\;\;\;
& \mapsto & 
\;\;\;\;\;\;\mbox{sgn}(\sigma) \, \sigma^*f.
\end{eqnarray}
Since the action of $\Sym_5$ is defined over $\Q$,
we infer the splitting 
\begin{eqnarray}
\label{eq:V^4}
T(S)=V^4
\end{eqnarray}
for some irreducible Hodge structure $V$.
Here $V$ can be found as $+1$-eigenspace in $T(S)$ for any element $\sigma\in\Sym_5$
such that $\dim H^{2,0}(S)^{\sigma^*=1}=1$.
For instance, using any involution of sign $-1$ in $\Sym_5$,
$V$ will appear on a K3 quotient $X$ of $S$
that we exhibit below.

\subsection{}
Recall the special member $S_a$.
Since a quintic $S$ has $b_2(S)=53$, we know by Lemma \ref{lem-lowerbound} that $T(S_a)$ has rank at most 12.
On the other hand, 
\[
\mbox{rank} (T(S_a))\geq 2p_g(S_a) =  8
\]
by Lefschetz' theorem.
The splitting $T(S_a)=V_a^4$ implies that $T(S_a)$ has rank 4-divisible.
Hence 
there are only two possibilities remaining: 
\begin{eqnarray}
\label{eq:rkT}
\mbox{rank}(T(S_a)) = 8 \text{ or } 12.
\end{eqnarray}

Our goal is to prove that the latter alternative holds:

\begin{Lemma}
\label{lem:T}
On $S_a$ over $\C$, the Hodge structure $T(S_a)$ has rank 12.
\end{Lemma}

We shall prove the lemma by constructing a suitable K3 quotient $X_a$ of $S_a$.
Before going into the details, we  comment briefly on other possible approaches.
In a similar situation of a surface with $\Sym_5$ action in \cite{vGS},
the authors alluded to modularity in order to rule out the analogous small rank alternative.
This line of argument does not apply here since $S_a$ is not defined over $\Q$.
Instead one can use the Artin-Tate conjecture to compare square classes of discriminants
of reductions modulo different primes.
For $S_a$, however, this approach would always result in perfect 4th powers
due to the splitting \eqref{eq:V^4}.
This is the main reason to switch to a quotient surface of $S_a$ that is a K3 surface
(or any other surface of geometric genus 1).

In order to prove Lemma \ref{lem:T} 
our first aim is to construct a quotient surface of $S$ that has geometric genus 1.
As indicated above,
the easiest way to achieve this builds on an involution interchanging exactly two homogeneous coordinates, say 
\begin{eqnarray*}
\imath: \;\;\;\;\; S\;\;\;\;\;\;\;\;  & \to & \;\;\;\;\;\;\;\;\;\; S\\
~[x_0,\hdots,x_4] & \mapsto & [x_1,x_0,x_2,x_3,x_4].
\end{eqnarray*}
Since $\imath^*$ fixes exactly one holomorphic 2-form on $S$ up to scaling,
we find 
\[
T(S)^{\imath^*=1}=V
\]
for the Hodge structure $V$ alluded to in \eqref{eq:V^4}.

The involution 
fixes the degree 5 curve $\{x_0=x_1\}$ on $S$ and the isolated point $[1,-1,0,0,0]$, 
yielding an $A_1$ singularity on the quotient surface $S/\imath$. 
We introduce the invariant coordinates
\[
u=x_0x_1,\;\;\; v=x_0+x_1.
\]
Then $S/\imath$ is birationally given by the resulting equation
 in weighted projective space $\PP[1,1,1,1,2]$
with weighted homogeneous coordinates $x_2,x_3,x_4,v,u$.
Expressing $x_4$ through $s_1$ and setting affinely $v=1$,
we obtain the affine equation
\begin{eqnarray*}
S/\imath:\;\;(x_3+1) (x_2+1) (x_2+x_3) (x_2^2+x_2 x_3+x_2+x_3+1+x_3^2) (\lambda +1)\\
 =(\lambda x_2 x_3-\lambda+\lambda x_2^2 x_3+\lambda x_2 x_3^2-1) u+(\lambda +1) u^2.
\end{eqnarray*}
This realises $S/\imath$ as  a double sextic with rational double point singularities
over the affine $x_2, x_3$-plane.
Hence its minimal projective resolution $X$ is a K3 surface.
By construction, $X$ carries the Hodge structure
\begin{eqnarray}
\label{eq:T(X)}
T(X)= V,
\end{eqnarray}
and the corresponding equality holds for  Galois representations.

We can also interpret $X$ as an elliptic surface over the affine $x_3$ line, say.
This fibration has eight obvious sections induced by the lines on $S_a$;
these are given by the two roots of $u$ at 
$x_2= 0, -1, -x_3, -1-x_3$.
Converting to Weierstrass form, we directly find a 2-torsion section; 
translation to $(0,0)$ yields the following equation
in the standard coordinates $x=x_2, t=x_3$:
\begin{eqnarray}
\label{eq:ell}
X: && 
u^2 = x (x^2+A(t)x+B(t))
\end{eqnarray}
\begin{eqnarray*}
A & = & 
\lambda^2 t^4-(4+8 \lambda+2 \lambda^2) t^3-(24 \lambda+12+11 \lambda^2) t^2\\
&& \;\;\;\;\;\;\;\; -4 (2 \lambda+3) (1+\lambda) t-4 (1+\lambda)^2 \\
B & = & {
16 t (t+1) (1+\lambda)^2 
[(2 \lambda+1) (t^4+t^3)+(3 \lambda+2) t^2+(2 \lambda+2) t+1+\lambda].}
\end{eqnarray*}
The discriminant reveals generally 6 singular fibres of type $I_2$ in Kodaira's notation
at the zeroes of $B$
and a split-multiplicative fibre of type $I_4$ at $\infty$. 

\subsection{Special K3 surface $X_a$}

We shall now specialise to the quintic $S_a$ and its K3 quotient $X_a$
where $a$ is
 given by \eqref{eq:a} as before.
By \eqref{eq:rkT} we know that the irreducible Hodge structure $V_a$ has rank 2 or 3,
so $X_a$ has Picard number $19$ or $20$ over $\C$ by \eqref{eq:T(X)}.

\begin{Proposition}
\label{prop}
The complex K3 surface $X_a$ has Picard number 19.
\end{Proposition}

In order to prove the proposition,
we assume on the contrary that $\rho(X_a)=20$ 
and establish a contradiction by reducing modulo different good primes
and comparing the square classes of the discriminants of 
the N\'eron-Severi lattices
by virtue of the Artin-Tate conjecture.
This method was pioneered in \cite{vL}, refined in \cite{RK} and applied in a similar context in \cite{S-NS}.

To be on the safe side when applying the reduction method,
we compute the primes of bad reduction for $X_a$ directly.
This is easily achieved thanks to the elliptic fibration 
which specialises from the pencil $X_\lambda$.
On $X_a$ it attains 8 singular fibres of type $I_2$, 
each of them defined over the ground field $k(a)$
(in addition to the $I_4$ fibre at $\infty$).
For the bad primes it suffices to study the degeneration of this fibration upon reduction mod $p$,
i.e.~whether the types of singular fibers change upon reduction.

\begin{Lemma}
\label{lem:bad}
$X_a$ has good reduction outside characteristics $\{2,3,5,11,17,433\}$.
\end{Lemma}

\begin{proof}
It is an easy exercise using the discriminant to verify that 
the types of singular fibres of
the given elliptic fibration 
do not change
outside the above characteristics and $83, 151$.
In the latter two characteristics (and exactly for $a\in\F_p$
coming from the $\F_p$-rational value of $b$ from \eqref{eq:a}), 
the fibration degenerates by merging fibres of type $I_1$ and $I_2$ to a fibre of type $III$.
In other words, the two nodes of the $I_2$ fibre come together
without reduction causing a singularity.
Thus $X_a$ has good reduction at all primes dividing $83$ and $151$,
and the lemma follows.
\end{proof}

On the quintic $S_a$, the above characteristics (except for $2$) are visible directly in terms of 
Lemma \ref{lem:smooth}:
the value of $a$ coming from the $\F_p$-rational root of \eqref{eq:a}
equals some exceptional parameter for $\lambda$ from the lemma.
This is all there is:

\begin{Corollary}
\label{cor:bad}
$S_a$ has good reduction outside characteristics $\{3,5,11,17,433\}$.
\end{Corollary}

\begin{proof}
By Lemma \ref{lem:smooth} it suffices to check when $a$ reduces to some 
singular parameter for $\lambda$.
That is, for each of these parameters over $\Q$, we substitute the resulting value of $b$
in the equation \eqref{eq:a}
and compute the corresponding primes as claimed.
At the primes  $\p\mid 2$,
we should note that any  $a$  from \eqref{eq:a} reduces to zero modulo $\p$.
Since $S_0$ is smooth outside characteristics $3,13,17$ by the Jacobi criterion,
this suffices to conclude the proof.
\end{proof}



\subsection{Proof of Proposition \ref{prop}}
As a preparation we recall the Lefschetz fixed point formula for $X_a$.
Over some finite field $\F_q$ ($q=p^e, p$ prime) containing a root $a$ from \eqref{eq:a},
it returns with some auxiliary prime $\ell$
\[
\# X_a(\F_q) = 1 + \mbox{tr Frob}_q^*(H_\text{\'et}^2(X_a\otimes\bar\F_p, \Q_\ell)) + q^2.
\]
On divisors, Frob$_q^*$ has eigenvalues $\zeta$ for
roots of unity $\zeta$.
In particular, the trace on the algebraic subspace  inside $H_\text{\'et}^2(X_a\otimes\bar\F_p, \Q_\ell(1))$ 
spanned by $\NS(X_a\otimes\bar\F_p)$ via the cycle class map
equals an integer.
Presently $\rho(X_a\otimes\bar\F_p)=20$ or $22$ by assumption, since $\rho=21$ is ruled out by \cite{A}.
By the above considerations, any non-congruence 
\begin{eqnarray}
\label{eq:nc}
\# X_a(\F_q)\not\equiv 1\mod q 
\end{eqnarray}
implies $\rho(X_a\otimes\bar\F_q)\leq 20$.
This non-congruence is easily verified at specific primes;
for instance, Table \ref{tab:traces} shows 
$\rho(X_a\otimes\bar\F_p)\leq 20$ for $p=19, 23$ and the respective choice of  solution to \eqref{eq:a} in $\F_p$.
Thus our assumption implies the equality $\rho(X_a\otimes\bar\F_p) = 20$, 
and in fact the validity of the Tate conjecture for $X_a$ over 
any finite extension of $\F_{19}$ and $\F_{23}$ 
(alternatively one can use the elliptic fibration with section on $X_a$ and appeal to \cite{ASD}).

Consider now some prime $p$ such as $p=19,23$ where we can prove by the above elementary means
that $\rho(X_a\otimes\bar\F_p)=20$.
Then the characteristic polynomial of Frob$_q^*$ on $H_\text{\'et}^2(X_a\otimes\bar\F_p, \Q_\ell)$ factors into
a product of cyclotomic polynomials (shifted by $q$) and a single quadratic factor
\begin{eqnarray}
\label{eq:a_q}
\mu_q(T) = T^2-a_qT+q^2.
\end{eqnarray}
where $a_q\equiv \#X_a(\F_p)-1\not\equiv 0\mod q$. 
Moreover $a_q\in\{-2q,\hdots,2q\}$ by the Weil conjectures.
Thus the parity of $\# X_a(\F_q)$ modulo $q$ predicts four possibilities for the  trace $a_q$
without any further knowledge about the Galois action on divisors.
(In fact the Galois action cannot be overly complicated 
since $S_a$ contains numerous non-trivial divisor classes over $\F_q$,
such as
all components of the 8 $I_2$ fibres and the  $I_4$ fibre at $\infty$ and the infinite section inherited from the generic member.)

Eventually, we want to apply the Artin-Tate conjecture \cite{T1} to $X_a$;
it is equivalent to the Tate conjecture by \cite{Milne}, so it holds in the present situation.
There is a little complication in mimicing the technique from \cite{vL}: 
the Artin-Tate conjecture for $X_a/\F_q$
allows us to read off the square class of the discriminant of $\NS(X_a\otimes\bar\F_p)$
from the characteristic polynomial $\mu_q(T)$ 
a priori only if $\NS(X_a\otimes\bar\F_p)$ is actually defined over $\F_q$,
i.e.~generated by divisors defined over $\F_q$.
Presently this need not hold over $\F_p$.
However, as $\mu_q(T)$ is quadratic,
there is a simple way to circumvent this problem and avoid computing explicitly the minimal extension $\F_q$
where $\NS(X_a)$ is defined.
For this purpose we introduce the following auxiliary general result.

\begin{Lemma}
\label{lem}
Let $X/\F_q$ be a K3 surface with geometric Picard number 20.
Consider the characteristic polynomial $\mu_q(T)$ as above.
Let $d\in\Z$ such that $\mu_q(T)$ splits in $\Q(\sqrt{d})$.
Then the square class of the discriminant of $\NS(X\otimes\bar\F_q)$ is given by $d$.
\end{Lemma}

\begin{proof}
Denote the roots of $\mu_q(T)$ by $\alpha, \bar\alpha$.
We will need that $\alpha$ does not equal $q$ times a root of unity.
Equivalently the Tate conjecture holds for $X$, as checked for $X_a$ in conjunction with \eqref{eq:nc}.
For arbitrary $X$, assume to the contrary that $\alpha$ takes the shape $q$ times a root of unity.
Then $X$ has infinite height, so it is supersingular in Artin's sense.
On the other hand, $X$ admits an elliptic fibration,
induced by a divisor class with square zero 
(this holds for any K3 surface with $\rho\geq 5$
since then $\NS$ represents $0$).
But then $\rho=22$ by \cite[Thm.~1.7]{A}, giving the required contradiction.

Next we claim that the splitting field of $\mu_q(T)$ is stable under extension.
To see this, we compute $\mu_{q^e}(T)=(T-\alpha^e)(T-\bar\alpha^e)$ for any $e\in\N$.
Then we use that $\alpha^e\not\in\Q$ by the above considerations.

As a consequence we can assume that $q$ is chosen in such a way 
that $\NS(X\otimes\bar\F_q)$ is already defined over $\F_{q^2}$,
so that $D=\mbox{disc}(\NS(X\otimes\bar\F_q))=\mbox{disc}(\NS(X\otimes\F_{q^2}))$.
Note that $D<0$ by the Hodge index theorem.
The Artin-Tate conjecture \cite{T1} then predicts that the square class of $-D$ is given by
$\mu_{q^2}(T)$ evaluated at $T=q^2$ up to a factor of $q$:
\begin{eqnarray}
\label{eq:sqcl}
2q^2-a_{q^2} = -M^2D.
\end{eqnarray}
Here $M^2$ is the size of the Brauer group of $X$ over $\F_{q^2}$.
Generally we have $a_{q^2}=a_q^2-2q^2$, so \eqref{eq:sqcl} simplifies as
\begin{eqnarray}
\label{eq:sqcl2}
4q^2-a_q^2 = -M^2D.
\end{eqnarray}
But this is equivalent to the splitting field of $\mu_q(T)$ being exactly $\Q(\sqrt{D})$.
\end{proof}



\begin{Remark}
As in \cite{S-NS} one can also deduce that $q$ splits into two principal ideals in $\Q(\sqrt{D})$.
In other words, if $q=p^e$, then the prime factors of $p$ have order dividing $e$ in the class group $\mbox{Cl}(\Q(\sqrt{D}))$ 
which gives a severe restriction on $e$.
\end{Remark}

Now let us return to our specific K3 surface $X_a$.
Counting points over $\F_p$ for $p=19,23$
we infer from Table \ref{tab:traces} that $\rho(X_a\otimes\bar\F_p)=20$ 
at both primes by the congruence argument from \eqref{eq:nc}.

\begin{table}[ht!]
$$
\begin{array}{c|c|c|c}
p & \# X_a(\F_p) & a_p & D\\
\hline

19 & 676 &                               29 &  -67\\                                     
&&                              10 & -21\\
&&                                -9 & -29\cdot 47\\
&& -28 & -3\cdot 5\cdot 11\\
\hline
23 & 924 & 26 & -10\\
&& 3 & -43\\
&&-20 & -3\cdot 11\cdot 13\\
&& -43 & -3\cdot 89
\end{array}
$$
\caption{Possible discriminants of $\NS(X_a\otimes\bar\F_p)$}
\label{tab:traces}
\end{table}

Recall the original assumption $\rho(X_a\otimes\bar\Q)=20$ and consider the isometric specialisation embedding
induced by reduction modulo some good prime $\p$:
\begin{eqnarray}
\label{eq:spec}
\NS(X_a\otimes\bar\Q) \hookrightarrow \NS(X_a\otimes\bar\F_\p). 
\end{eqnarray}
Presently our assumption implies that at $p=19, 23$ the embedding \eqref{eq:spec} has finite cokernel.
In consequence, the square classes of all three N\'eron-Severi lattices 
under consideration coincide.
But then by Table \ref{tab:traces} this is impossible for $p=19$ and $23$ 
thanks to Lemma \ref{lem}
since no two possibilities for $D$ match.
Hence we reach the desired contradiction.
This concludes the proof of Proposition \ref{prop}.
\qed

\subsection{Proof of Lemma \ref{lem:T} and Theorem \ref{thm}}
From Proposition \ref{prop} together with the splitting \eqref{eq:V^4}
we directly deduce Lemma \ref{lem:T}.
Theorem \ref{thm} follows immediately in conjunction with Lemma \ref{lem-lowerbound}.
\qed

\subsection{Remark on the Tate conjecture for the pencil $\{S_\lambda\}$}

It is common to infer the Tate conjecture for a surface from its validity for some cover (cf.~\cite{Milne}).
Here we reverse this argument and verify the Tate conjecture for the quintics $S$
through K3 quotients.
A similar technique was applied in \cite{SS}, but the situation here is more complicated
since the surfaces in question have different geometric genus.

\begin{Proposition}
\label{prop:tate}
The Tate conjecture holds true for any smooth quintic $S$ in the pencil $\{S_\lambda\}$
over any finite field.
\end{Proposition}

\begin{proof}
Let $k$ denote some finite field such that $S\otimes k$ is smooth. 
We shall make use of the Galois-equivariant comparison isomorphism
\begin{eqnarray}
\label{eq:comparison}
H_\text{\'et}^2(S\otimes \bar \Q, \Q_\ell) \cong H_\text{\'et}^2(S\otimes \bar k, \Q_\ell).
\end{eqnarray}
In view of the decomposition of Galois representations
\begin{eqnarray}
\label{eq:decomp}
H_\text{\'et}^2(S\otimes \bar \Q, \Q_\ell) = (T(S)\otimes\Q_\ell) \oplus (\NS(S)\otimes\Q_\ell(-1)),
\end{eqnarray}
the Tate conjecture is valid on the image of $\NS(S)\otimes\Q_\ell(-1)$.
By \eqref{eq:decomp} it remains to study the image of $T(S)\otimes\Q_\ell$ inside 
$H_\text{\'et}^2(S\otimes \bar k, \Q_\ell)$.
The main idea here is that this possibly non-algebraic part is governed by K3 surfaces
as follows.

By \eqref{eq:sigma},
there is no non-zero 2-form in $H^{2,0}(S)$ anti-invariant under each involution in $\Sym_5$
of sign $-1$.
In detail, it suffices to consider 
the involutions interchanging $x_0, x_i$ for $i=1,\hdots,4$.
It follows that the invariant Galois representations $V\otimes\Q_\ell$
for these 4 involutions cover all of $T(S)\otimes\Q_\ell = (V\otimes\Q_\ell)^4$.
In consequence, all of $T(S)\otimes\Q_\ell$ can be realised via pull-back
from the corresponding K3 quotients which are all isomorphic to $X$.
%
%
As this K3 surface admits an elliptic fibration with section \eqref{eq:ell},
the Tate conjecture holds for $X$ by \cite{ASD}.
Pulling back divisors from $X$ to $S$ via the various quotient maps, 
we infer that the Tate conjecture holds for $S$.
\end{proof}

\begin{Remark}
The same argument works for the desingularisation of any singular irreducible member of the pencil
which lifts to a quintic over $\bar\Q$ with the same configuration of singularities.
The only quintics where this fails ($(\lambda,p)=(0,3), (4,11), (6,13)$)
can be covered by hand.
Subsequently the Tate conjecture can also be verified for the singular members of the pencil themselves
where we compare the Picard number with the number of  poles of the zeta function 
in its original definition as exponential sum involving numbers of points
\end{Remark}

We can use the Tate conjecture to compute the Picard number of $S_a$ in any characteristic.
In general, there are two alternatives for the Picard number as we indicate below. 
Here we only have to rule out that $S_a$ becomes reducible mod $\p$
(for the singular case see Remark \ref{rem:tate}).
This happens exactly in characteristic $5$ for the $\F_5$-rational root of \eqref{eq:a}.
Characteristic $2$ also plays a special role, as we shall exploit in Section \ref{s:NS}.

\begin{Corollary}
\label{cor:tate}
Let $p\neq 2$ be a prime and $a\in\F_q\subset\bar\F_p$ given by a root of \eqref{eq:a} such  that 
$S_a\otimes\bar\F_p$ is smooth.
Then
\[
\rho(S_a\otimes\bar\F_p) =
\begin{cases}
45, \text{ if } \# S_a(\F_q)\not\equiv 1 \mod q,\\
53, \text{ if } \# S_a(\F_q)\equiv 1 \mod q.
\end{cases}
\] 
\end{Corollary}

\begin{proof}
Since the Tate conjecture holds for $S_a/\F_q$,
it suffices to compute the characteristic polynomial $\chi_q(T)$ of Frob$_q$ 
on $H_\text{\'et}^2(S_a\otimes \bar\F_p, \Q_\ell)$.
Presently we have
\[
\chi_q(T) = (T-q)^{40} (T\mp q) \chi_q'(T)^4
\]
where the first two factors come from the lines and the extra generator of $\NS(S_a\otimes\bar\Q)$ 
and the last corresponds to $T(S_a)$.
That is, the degree 3 polynomial $\chi_q'(T)$ comes from the motive $V$ of the K3 surface $X_a$.
Thus it takes the shape
\[
\chi_q'(T) = (T\mp q) (T^2 -a_q T \pm q^2) 
\]
where the sign alternative $-q^2$ may only persist if $a_q=0$. In particular,
\begin{eqnarray}
\label{eq:X-S}
\rho(X_a\otimes\bar\F_p) = 
\begin{cases}
20, \text{ if } a_q\not\equiv 0 \mod q,\\
22, \text{ if } a_q\equiv 0 \mod q.
\end{cases}
\end{eqnarray}
By Proposition \ref{prop:tate} the corresponding statement for $S_a$ reads
\[
\rho(S_a\otimes\bar\F_p) = 
\begin{cases}
45, \text{ if } a_q\not\equiv 0 \mod q,\\
53, \text{ if } a_q\equiv 0\mod q.
\end{cases}
\] 
In order to translate to the number of points, we apply the Lefschetz fixed point formula to find
\begin{eqnarray}
\label{eq:pm}
\# S_a(\F_q) = 1 + 40q \pm q + 4(a_q\pm q) + q^2.
\end{eqnarray}
Outside characteristic $2$, the congruence for $a_q$ is equivalent to that for $\# S_a(\F_q)$ from the corollary.
%
\end{proof}


\begin{Remark}
\label{rem:tate}
With minor modifications the same arguments apply
to the desingularisations of the singular irreducible quintics $S_a/\F_p$
at $p\in\{3,11,17,433\}$.
\end{Remark}

In practice it is often easier to use the condition \eqref{eq:X-S} involving the quotient K3 surface.
For instance,  it follows directly from Table \ref{tab:traces} that
\[
\rho(S_a\otimes\bar\F_p) = 45 \;\;\; \text{ for } \;\; p=19, 23.
\]
In fact, the condition on the K3 quotient can also be used in characteristic $2$
where \eqref{eq:pm} does not prove useful because of the extra factor of $a_q$.
Alternatively we can pursue a different approach in characteristic $2$
based on the fact that $S_a$ reduces to $S_0$ modulo $2$.
In consequence there are additional lines on the reduction.
This approach will feature in the next section
as it proves very useful for the computation of the actual lattice $\NS(S_a)$ up to small index.

%

\section{$\NS(S_a)$ up to index $16$}
\label{s:NS}

We conclude the paper by computing $\NS(S_a)$ up to finite index (actually a $2$-power at most $16$).
To this end we are concerned with the $R^*$-invariant divisor class $D$
from Section \ref{s:Godeaux}
which complements the 75 lines to generate $\NS(S_a)$ over $\Q$ by the proof
of Lemma \ref{lem-lowerbound}.
In fact, though we will not exhibit this divisor class on $S_a$ over $\C$,
we will still pursue an explicit approach suggested by Reid in \cite{reid},
albeit in a highly degenerate situation in positive characteristic.

Following \cite{reid},
the main idea to find $D$ is to consider elliptic curves of degree $5$ on $S_a$
(or generically on $S_\lambda$)
which are invariant under the order 5 automorphism $R$.
These elliptic curves are given as suitable  intersection of the cubics
o which $R^*$ acts by the fifth roots of unity (cf.~\cite[p.~362]{reid}).
Presently, the computations become too involved not only on the generic surface $S_\lambda$,
but also on $S_a$ itself.
We remedy this by considering the reduction of $S_a$ mod $2$ which is $S_0\otimes\F_2$.
Here the system of equations simplifies enough to solve them directly.
However, the degree $5$ divisors $D$ thus obtained are reducible
as they decompose into 5 lines on $S_0\otimes\F_{4}$ which separately
only lift to $S_{-2}$ over $\C$.

Indeed in characteristic $2$, the quintic $S_a\otimes\F_{4}$ contains 60 additional lines.
Following \cite{xie2010} these can be given as $\Sym_5$-orbits of
\[
x_0+x_1=x_2+\omega x_3=0
\]
and
\begin{eqnarray}
\label{eq:extra}
\ell_2:\;\;\; x_0+x_1-\omega^2 x_4=x_0+x_2-\omega x_3=0
\end{eqnarray}
where $\omega$ denotes a primitive third root of unity.
With a machine it is easily verified 
that the Gram matrix of the 135 lines in total has rank $53$.
Thus $S_a$ is supersingular in characteristic $2$, and in fact $\rho(S_a\otimes\F_{16})=53$.

Letting $R$ denote a permutation of order 5 in $\Sym_5\subset\Aut(S_a)$ as before,
we define an $R$-invariant divisor on $S_0\otimes\F_4$
by
\[
D_2 = \sum_{i=0}^4 R^i \ell_2.
\]

\begin{Lemma}
\label{lem:lift}
The divisor class of $D_2$ in $\NS(S_0\otimes\bar\F_2)$
lifts to $S_a$ to generate $\NS(S_a)\otimes\Q$ together with the 75 lines.
\end{Lemma}

\begin{proof}
Essentially the lemma amounts to a computation on the Godeaux surface $X_a\otimes\F_{16}$.
Consider the subgroup $N'\subset\NS(X_a\otimes\F_{16})$
generated by  $\pi(D_2)$ and  the 15 images of the 75 lines on $S_a\otimes\F_{16}$
specialised from characteristic zero.
An easy computation reveals that
the Gram matrix of $N'$ has rank $9$ and determinant $1$.
Hence $N'$ equals $\NS(X_a\otimes\bar\F_2)$ up to torsion,
and in particular $D_2$ is independent of the 75 lines in $\NS(S_a\otimes\bar\F_2)$.

For the lifting, consider the commutative diagram
$$
\begin{array}{ccc}
\NS(X_a) &  \hookrightarrow
& \NS(X_a\otimes\bar\F_2)\\
\downarrow && \downarrow\\
\NS(S_a) & \hookrightarrow
& \NS(S_a\otimes\bar\F_2)
\end{array}
$$
where the horizontal embeddings are induced by reduction mod $2$
and the vertical maps are given by pull-back.
By Poincare duality, the lattice 
\[
\Num(X_a) = \NS(X_a)/(\mbox{torsion})
\]
is unimodular. 
Since $N'$ induces a unimodular sublattice of $\Num(X_a\otimes\bar\F_2)$, 
we deduce that the top line embedding is in fact an isomorphism 
(a priori up to torsion).
But then the divisor class $D_2\in\NS(S_a\otimes\bar\F_2)$ lifts to $S_a$ the long way around the diagram:
via its image $\pi(D_2)$ in $X_a\otimes\F_4$, the above isomorphism and pull-back.
\end{proof}

Denote the lift of $D_2$ in $\NS(S_a)$ by $D$.
Together with the 75 lines $D$ generates a sublattice $M'\subset\NS(S_a)$ of finite index by Lemma \ref{lem:lift}. 
We claim that this is  at most a small $2$-power away from the full N\'eron-Severi lattice:

\begin{Proposition}
$[\NS(S_a):M']=2^i$ for some $i\in\{0,\hdots,4\}$. 
\end{Proposition}

\begin{proof}
One easily verifies (on $S_a\otimes\F_{16}$\,!) that $M'$ has rank $41$ and discriminant 
\[
\mbox{disc}(M')=2^{8}\cdot 3^4\cdot 5\cdot 11^4.
\]
Hence there could only be $2, 3$ or $11$-divisibility in $M'$.
To rule out the latter two alternatives,
we apply the supersingular reduction technique developed in \cite{SSvL} at $p=2$.
In the sequel we give a brief sketch of the argument.

We start with a $\Z$-basis $\mathcal B$ of $M'$,
considered on $S_a\otimes\bar\F_2$ by reduction.
Then we supplement $\mathcal B$ by 12 additional lines on $S_a\otimes\F_4$
for a $\Q$-basis $\mathcal B_2$ of $\NS(S_a\otimes\bar\F_2)$.
This furnishes us with a sublattice $M_2\subset \NS(S_a\otimes\bar\F_2)$
of rank $53$ and  discriminant $2^{16}\cdot 5^2$.
Immediately this shows that there cannot be any  elements in $M'$ which become $3$ or $11$-divisible in $\NS(S_a)$
since then these primes would necessarily appear in the discriminant of $M_2$ as well.
Hence $[\NS(S_a):M']=2^i$ for some $i\leq 4$.
\end{proof}

\subsection*{Acknowledgements}
S. Rams would like to thank Prof.~W.~P.~Barth for inspiring discussions on
quintics in autumn 2000.
We are grateful to the referee for his corrections and helpful comments.
This project was started during the second author's visit to Jagiellonian University in March 2011.
Thanks for the great hospitality, especially to S.~Cynk.

\end{document}